\documentclass[10pt,a4paper]{article}
\usepackage{amsmath,amsthm,amsfonts,amssymb}

\usepackage{mathtools} %extension of amsmath, provides bug fixes and other tools 
\usepackage{mathrsfs} %Provides a nice \mathscr command

\usepackage{color}
\usepackage{amssymb} 

\newtheorem{thm}{Theorem}

\newtheorem{cor}[thm]{Corollary}
\newtheorem{prop}[thm]{Proposition}

\begin{document}

\title{Optimal proper connection of graphs}

\author{Shinya Fujita\\ 
School of Data Science, Yokohama City University,\\
22-2 Seto, Kanazawa-ku, Yokohama 236-0027, Japan,\\
Email: shinya.fujita.ph.d@gmail.com\\
}

%\subjclass[2010]{05C15}

\date{}
%\subjclass[2000]{Primary 05C70}

\maketitle
\makeatother

\begin{abstract}
An edge-colored graph $G$ is \textit{properly colored} if no two adjacent edges share a color in $G$.
An edge-colored connected graph $G$ is \textit{properly connected} if between every pair of
distinct vertices, there exists a path that is properly colored.
In this paper, we discuss how to make a connected graph properly connected efficiently. 
More precisely, we consider the problem to convert a given monochromatic graph into properly connected by recoloring $p$ edges with $q$ colors so that $p+q$ is as small as possible. We discuss how this can be done efficiently for some restricted graphs, such as trees, complete bipartite graphs and graphs with independence number $2$. 
%Insert your abstract here. Include keywords, PACS and mathematical
%subject classification numbers as needed.
%\keywords{Optimal proper connection number \and Proper connection \and Edge-colored graph}
\end{abstract}

%%%%%%%%%%%%%%%%%%%%%%%%%%%%%%%%%%%%%%%%%%%%%%%%%%%%%%%%%%%%%%%%%%%%%%%%%%%%%%%%%%%%%%%%%%%%%%%%%%%%%%%%%%%%%%%%%%%%%%%%
%%%%%%%%%%%%%%%%%%%%%%%%%%%%%%%%%%%%%%%%%%%%%%%%%%%%%%%%%%%%%%%%%%%%%%%%%%%%%%%%%%%%%%%%%%%%%%%%%%%%%%%%%%%%%%%%%%%%%%%%
%%%%%%%%%%%%%%%%%%%%%%%%%%%%%%%%%%%%%%%%%%%%%%%%%%%%%%%%%%%%%%%%%%%%%%%%%%%%%%%%%%%%%%%%%%%%%%%%%%%%%%%%%%%%%%%%%%%%%%%%
\section{Introduction}\label{sec1}
%%%%%%%%%%%%%%%%%%%%%%%%%%%%%%%%%%%%%%%%%%%%%%%%%%%%%%%%%%%%%%%%%%%%%%%%%%%%%%%%%%%%%%%%%%%%%%%%%%%%%%%%%%%%%%%%%%%%%%%%
%%%%%%%%%%%%%%%%%%%%%%%%%%%%%%%%%%%%%%%%%%%%%%%%%%%%%%%%%%%%%%%%%%%%%%%%%%%%%%%%%%%%%%%%%%%%%%%%%%%%%%%%%%%%%%%%%%%%%%%%
%%%%%%%%%%%%%%%%%%%%%%%%%%%%%%%%%%%%%%%%%%%%%%%%%%%%%%%%%%%%%%%%%%%%%%%%%%%%%%%%%%%%%%%%%%%%%%%%%%%%%%%%%%%%%%%%%%%%%%%%

All graphs considered in this paper are finite and simple. Our notation in this paper is standard. 
For a graph $G=(V(G), E(G))$, let $\alpha(G)$ be the independence number of $G$. Also, let $\alpha'(G)$ be the size of a maximum matching of $G$. Let $\kappa(G)$ be the vertex-connectivity of $G$. Let $diam(G)$ be the diameter of $G$. For a vertex $x\in V(G)$, let $N_G(x)=\{z\in V(G)|\ xz\in E(G)\}$. For a vertex subset $S$ of $V(G)$, $G[S]$ stands for an induced subgraph of $G$ induced by $S$.  
For other terminology and notation not defined here, we refer the reader to \cite{West}

An edge-colored graph $G$ is \textit{properly colored} if no two adjacent edges share a color in $G$. Properly colored paths and cycles appear in a variety
of fields such as genetics \cite{Dorninger,Dorninger2} and social sciences \cite{Chou}.
An edge-colored connected graph $G$ is \textit{properly connected} if between every pair of
distinct vertices, there exists a path that is properly colored.  In \cite{Borozan}, Borozan et al. defined a new notion called the \textit{proper connection number} $pc(G)$ of a connected graph $G$, where $pc(G)$ is the minimum
number of colors needed to color the edges of $G$ to make it properly connected.
As described in \cite{HLQ,Li}, this concept has a real application to build an efficient  communication network with no radio-frequency interference between each pair of wireless signal towers. 
There, roughly speaking, to avoid interference, it is important not to share the same frequency when a wireless transmission passes through a signal tower. In fact, the proposed network model can be regarded as an edge-colored graph that is properly connected. 
%One can regard a signal tower as a vertex of an edge-colred graph, and associate an assigned color on each edge with the assigned frequency used for the communication between wireless signal towers.
 For a more precise description, see \cite{HLQ,Li}. 

%When building a communication network between wireless signal points, one fundamental
%requirement is that the network is connected. If there cannot be a direct connection between
%two points $X$ and $Y$, say for example if there is some obstacle in between, there must be a route
%through other points to get from $X$ to $Y$. As a wireless transmission passes through a signal
%point, to avoid interference, it would help if the incoming signal and the outgoing signal
%do not share the same frequency. Suppose that we assign a vertex to each signal point, an edge
%between two vertices if the corresponding signal points are directly connected by a signal, 
%and assign a color to each edge based on the assigned frequency used for the communication.
%Then the number of different frequencies needed to assign the connections between points so that
%there is always a path avoiding interference between each pair of points is precisely the
%proper connection number of the corresponding graph.
%Aside from this application, properly colored paths and cycles appear in a variety
%of other fields such as genetics \cite{Dorninger,Dorninger2} and social sciences \cite{Chou}.
Recently, the notion of proper connection number attracts much attention from both theoretical and practical aspects, and thus a lot of work has been done extensively (see e.g., \cite{BD,F,GL,HL,LW}). For details in this recent topic, we refer the reader to the nice  survey of Li and Magnant \cite{Li}.   

In this paper, we are concerned with making an edge-colored graph properly connected efficiently. 
Let $(G, c)$ be a connected graph with a given edge-coloring $c$. 
Now we consider how to make $(G,c)$ properly connected by recoloring some edges with some colors. To minimize our effort to make $G$ properly connected, it would be natural to focus on the minimum value on the sum of numbers of edges and colors among such  recolorings.
Note that, such a value should be zero when $(G, c)$ is already properly connected. 

Perhaps the most fundamental and laborious case to this problem would be the case where $c$ assigns a common color on every edge of $G$, that is, the case where $G$ is a monochromatic colored graph. 
Therefore, in this paper, we shall initiate this study by assuming that all edges of $G$ have already been colored by a common color, say color $0$. 
For an integer $i\neq 0$, color $i$ is called a \textit{new color}.  

Keeping this assumption in mind, we define the following cost function of edge-colored graphs called the \textit{optimal proper connection number} for a monochromatic connected graph $G$. 

$pc_{opt}(G):=\min\{p+q|$ we can make $G$ properly connected\\
\hfill by recoloring $p$ edges of $G$ with $q$ new colors$\}.$\\

For a monochromatic connected graph $G$, suppose that $G$ becomes properly connected by recoloring $p$ edges of $G$ with $q$ new colors such that $p+q=pc_{opt}(G)$. Then we call such an edge-coloring of $G$ an \textit{optimal recoloring of $G$}.

By definition, note that $pc_{opt}(K_n)=0$ holds because any monochromatic complete graph is properly connected. Indeed, we see that a graph $G$ satisfies $pc_{opt}(G)=0$ if and only if $G$ is isomorphic to a complete graph. 
We can easily determine $pc_{opt}(G)$ for small graphs and some basic family of graphs. 
For example, we can check that $pc_{opt}(K_{2,3})=pc_{opt}(K_{3,3})=3, pc_{opt}(K_{3,4})=pc_{opt}(K_{4,4})=4$, $pc_{opt}(K_{1,m})=2m-2$, $pc_{opt}(C_n)=pc_{opt}(P_n)=\lfloor (n-1)/2\rfloor+1$.\\

In this paper, we shall investigate graphs with a small optimal proper connection number.
Along this line, we give an upper bound on $pc_{opt}(G)$ when $G$ is a graph with $\alpha(G)\le 2$ (see Theorem~\ref{a2}). 
We also give a formula in terms of the optimal proper connection number for trees and complete bipartite graphs (see Theorems~\ref{kmn} and ~\ref{main}). 

This paper organizes as follows. In Section 2, we give some basic observation on optimal proper connection number in edge-colored graphs. In Section 3, we prove our main results (Theorems~\ref{a2}, \ref{kmn} and \ref{main}). In Section 4, we discuss some extension and open problems in this topic.

\section{Preliminaries}

In order to give a good upper bound on $pc_{opt}(G)$ for a monochromatic connected graph $G$, we start with the following basic observation.

\begin{prop}\label{f1}
If a monochromatic graph $G$ contains $H$ as a spanning connected subgraph, then $pc_{opt}(G)\le pc_{opt}(H)$. 
\end{prop}

\begin{proof}
The assertion obviously holds because an optimal recoloring of $H$ in $G$ assures us that $G$ is properly connected. 
\end{proof}

For a monochromatic graph $G$, an edge $e=x_1x_2\in E(G)$ is \textit{good} if $V(G)-(N_G(x_1)\cap N_G(x_2))$ can be partitioned into two parts $A_1$ and $A_2$ such that $A_i\subset N(x_i)$ and $G[A_i]\cong K_{|A_i|}$ for $i=1,2$ (possibly, $A_i=\emptyset$ for some $i$).      

We can characterize monochromatic graphs $G$ having $pc_{opt}(G)\le 2$ as follows. 

\begin{prop}\label{f2}
A monochromatic connected graph $G$ of order at least $3$ has $pc_{opt}(G)\le 2$ if and only if $G$ contains a good edge. Moreover, recoloring any good edge of $G$ with a new color can be an optimal recoloring of $G$. 
\end{prop}
\begin{proof}
The assertion obviously holds if $G$ is a complete graph. So we may assume that $G$ is not a complete graph. 
Suppose that $G$ contains a good edge $e$. Then, recoloring $e$ with color $1$, we can easily check that $G$ is properly connected by the definition of a good edge. Thus $pc_{opt}(G)\le 2$ and the second assertion holds. 

Next suppose that $pc_{opt}(G)\le 2$ and consider an optimal recoloring of $G$. 
We may assume that $G$ is now properly connected and $G$ has exactly one edge $e=x_1x_2$ with color $1$ and all other edges have color $0$. We claim that $e$ is a good edge. 
If there exists a vertex $y\in V(G)-(N_G(x_1)\cup N_G(x_2))$, then there is no properly colored path joining $x_1$ and $y$ in $G$ because $N_G(y)\cap \{x_1, x_2\}=\emptyset$ and $e$ is a unique edge having color $1(\neq 0)$. This contradicts the assumption that $G$ is properly connected. Thus we may assume that  $V(G)-(N_G(x_1)\cap N_G(x_2))$ can be partitioned into two parts $A_1$ and $A_2$ such that $A_i\subset N(x_i)$ for $i=1,2$  (possibly, $A_i=\emptyset$ for some $i$). 
Suppose that $A_1$ contains two vertices $y,z$ such that $yz\notin E(G)$. 
Since $\{y,z\}\cap N_G(x_2)=\emptyset$, obviously there is no properly connected path joining $y$ and $z$ in $G$. Again, this contradicts the assumption that $G$ is properly connected. Thus, by the symmetry of the roles of $A_1$ and $A_2$, we see that $G[A_i]\cong K_{|A_i|}$ for $i=1,2$ (possibly, $A_i=\emptyset$ for some $i$). 
Hence $e$ is a good edge, as claimed. This completes the proof of Proposition~\ref{f2}.     

\end{proof}

We next consider monochromatic connected graphs $G$ with $pc_{opt}(G)\le 3$. 
Unlike the case where $pc_{opt}(G)\le 2$, it seems complicated to characterize those graphs. As an initial step, in this paper, we investigate what kind of graphs $G$ satisfy $pc_{opt}(G)\le 3$. 

\begin{prop}\label{f3}
For a monochromatic graph $G$,  if $G$ contains a complete bipartite graph $H$ such that $H$ is a spanning subgraph of $G$ and each partite set of $H$ contains an edge in $G$, then $pc_{opt}(G)\le 3$. 
\end{prop}
\begin{proof}
Let $A_1$ and $A_2$ be the partite sets of $H$. 
By assumption, let $e_i=x_iy_i$ be an edge of $G[A_i]$ for $i=1,2$. 
Note that $e_i$ is a good edge of $G[A_{3-i}\cup\{x_i, y_i\}]$ for $i=1,2$. 
By Proposition~\ref{f2}, $pc_{opt}(G[A_{3-i}\cup\{x_i, y_i\}])\le 2$ and recoloring $e_i$ with color $1$, $G[A_{3-i}\cup\{x_i, y_i\}]$ is properly connected for $i=1,2$. 
Since $xy\in E(G)$ for any pair $x\in A_1, y\in A_2$, this implies that recoloring both $e_1$ and $e_2$ with color $1$ makes $G$ properly connected. This shows that $pc_{opt}(G)\le 3$. 

\end{proof}

By Proposition~\ref{f3}, we see that a monochromatic complete multipartite graph has a small optimal proper connection number because it contains the graph described in Proposition \ref{f3} as a spanning subgraph. 

\begin{cor}
Let $G$ be a monochromatic complete multipartite graph such that $G\cong K_{n_1,n_2,\ldots, n_l}$, where $l\ge 3$ and $2\le n_1\le \ldots \le n_l$. Then $pc_{opt}(G)\le 3$. 
\end{cor}

We finally give some observation on graphs with a forbidden subgraph condition. 
A graph $G$ is \textit{$P_4$-free} if it contains no $P_4$ as an induced subgraph. 
Although the following proposition has nothing to do with edge-coloring of graphs, it is useful when we prove our main results. 

\begin{prop}\label{P4free}
If $G$ is a connected $P_4$-free graph, then $G$ is a complete graph or $G$ contains a spanning complete bipartite subgraph $K_{m,n}$ with $\min\{2,\kappa(G)\}\le n\le m$ such that one partite set of the $K_{m,n}$ forms a minimum cutset of $G$. 
\end{prop}
\begin{proof}
We may assume that $G$ is not a complete graph.  
Let $S$ be a minimum cutset of $G$, and let $C_1,\ldots, C_l$ be the components of $G-S$, where $l\ge 2$.  
It suffices to show that every vertex of $S$ sends edges to all the vertices of $G-S$. 
Suppose not, and take $x\in S$ and $y\in V(G-S)$ such that $xy\notin E(G)$. 
We may assume that $y\in V(C_1)$. Since $S$ is a minimum cutset of $G$, there exists $z\in V(C_1-y)$ and $w\in V(C_2)$ such that $xz, xw\in E(G)$. Since $C_1$ is a component, we may assume that $y$ and $z$ are chosen so that $yz\in E(G)$ (to see this, take a shortest path $P=xx_1x_2\ldots x_ly$ in $G[\{x\}\cup C_1]$; if $l\ge 2$ then reset $y,z$ as $x_1=z, x_2=y$; if $l=1$ then reset $z$ as $x_1=z$).
Now, the path $yzxw$ is an induced $P_4$ in $G$. This is a contradiction. Thus the assertion holds. 

\end{proof}

\section{Main results}
%For a graph $G$, let $\alpha(G)$ be the independence number of $G$. 

%Our main result in this direction is following. 

We firstly give a sharp upper bound on $pc_{opt}(G)$ for graphs with small independence number. 

\begin{thm}\label{a2}
If $G$ is a monochromatic connected graph of order $n\ge 1$ such that $\alpha(G)\le 2$ then $pc_{opt}(G)\le 3$. 
\end{thm}

\begin{proof}

%We prove the theorem by induction on $n$.
The theorem obviously holds for small $n$. 
So we may assume that $n\ge 5$. 
If $\alpha(G)=1$ then $pc_{opt}(G)=0$ since $G$ is a complete graph. Thus we may assume that $\alpha(G)=2$ and hence $G$ is not a complete graph.

We firstly consider the case where $G$ is a $P_4$-free graph.  
In view of Proposition~\ref{P4free}, let $S$ be a minimum cutset of $G$ such that $G$ contains a spanning complete bipartite graph whose partite sets are $S$ and $V(G-S)$. Since $\alpha(G)=2$, $G-S$ consists of two components $C_1,C_2$ such that each $G[C_i]$ forms a complete graph. 
Suppose for the moment that $|S|=1$, say $S=\{v\}$. 
Take $u\in V(C_1)$. Note that, by the structure of $G$, $uv$ is a good edge of $G$. Consequently, by Proposition~\ref{f2}, $pc_{opt}(G)\le 2$.

Thus we may assume that $|S|\ge 2$. 
%Again, by Proposition~\ref{P4free}, note that $G$ contains a complete bipartite graph whose partite sets are $S$ and $G-S$ as a spanning subgraph. 
In view of Proposition~\ref{f3}, if both $G[S]$ and $G-S$ contain an edge, respectively, then $pc_{opt}(G)\le 3$. Thus we may assume that either $G[S]$ or $G-S$ has no edge. 
Since $\alpha(G)=2, n\ge 5$ and $S$ is a minimum cutset of $G$, it suffices to consider the case where $G[S]$ has no edge and $|S|=2$.  Let $e_1, e_2$ be two independent edges joining a vertex of $S$ and a vertex of $V(G-S)$, respectively.  Recoloring both $e_1$ and $e_2$ with color $1$, we can check that $G$ is properly connected because $C_1$ and $C_2$ are complete graphs. Thus $pc_{opt}(G)\le 3$.

Hence we may assume that $G$ contains a path $P=p_1p_2p_3p_4$ such that $G[\{p_1,p_2,p_3,p_4\}]\cong P_4$.  
Since $\alpha(G)=2$, we can partition $V(G-P)$ into three parts $X_1,X_2,X_3$ such that $X_1=\{v\in V(G-P)|\ vp_1, vp_2\in E(G)\}$, 
$X_2=\{v\in V(G-P)|\ vp_3, vp_4\in E(G)\}$ and $X_3=\{v\in V(G-P)|\ vp_1, vp_4\in E(G)\}$.  
Then, by recoloring $p_1p_2$ and $p_3p_4$ with color $1$, respectively, we can easily check that $G$ is properly connected, thereby proving that $pc_{opt}(G)\le 3$.

\end{proof}

The upper bound on $pc_{opt}(G)$ in Theorem~\ref{a2} is sharp. To see this, let $A_0, A_1, \ldots, A_4$ be disjoint cliques, and add all edges between $A_i$ and $A_{i+1}$, where the indices are taken modulo $5$. Let $G$ be the resulting graph.  
In order to make $G$ properly connected, obviously we need to recolor at least two edges with a new color. 

So far, we observed graphs with small upper bound on $pc_{opt}(G)$. 
Now we consider a question to ask what kind of a family of graphs $\mathcal{G}$ we can describe an equality $pc_{opt}(G)=f(G)$ for each $G\in \mathcal{G}$, where $f(G)$ is a certain value depending on some parameters of $G$. 
In fact, complete bipartite graphs and trees belong to such family of graphs. 

\begin{thm}\label{kmn}
Let $G$ be a monochromatic complete bipartite graph $K_{m,n}$ such that $m\ge n\ge 2$ and $m+n\ge 9$. Then $pc_{opt}(G)=4$ for $n=2, 3$, and $pc_{opt}(G)=5$ for $n\ge 4$.  
\end{thm}

\begin{proof}
Let $N, M$ be the partite sets of $G(=K_{m, n})$ with $n=|N|, m=|M|$. 
We first consider the case $2\le n\le 3$. Take $a,b \in N$.

The upper bound $pc_{opt}(G)\le 4$ can be obtained by recoloring $ax$ with color $1$ and $bx$ with color $2$, where $x\in M$. (Since $|N-\{a,b\}|\le 1$, it is easy to check that the resulting edge-colored $G$ is properly connected.) 
Toward a contradiction, suppose that $pc_{opt}(G)\le 3$. This implies that we have exactly one new color (say, color $1$) to recolor at most two edges in $G$ to make $G$ properly connected. We can easily check that $G$ cannot be properly connected if we recolor exactly one edge. Thus we may assume that exactly two edges, say $e_1, e_2\in E(G)$ are recolored with color $1$ so that $G$ is properly connected.  
If $e_1$ and $e_2$ share a vertex, then we can check that there is no properly colored path joining other two vertices of $e_1$ and $e_2$.
Thus $e_1$ and $e_2$ must be a matching in $G$. 
Since $n\le 3$ and $m+n\ge 9$, there exist two vertices $u, v\in V(G)$ with $uv\notin E(G)$ such that neither $u$ nor $v$ is on $e_i$ for $i=1,2$. 
Since $G$ is a complete bipartite graph, we see that there is no properly colored path joining $u$ and $v$, a contradiction. 
Thus we have $pc_{opt}(G)=4$. 
 
We next consider the case $n\ge 4$. 
Suppose that $pc_{opt}(G)\le 4$. 
To give an optimal recoloring of $G$, let us firstly consider the case we will use exactly one new color, say color $1$. 
In that case, we can recolor at most three edges. Since $n+m\ge 9$, wherever we rcolor at most three edges in $G$, we can find two vertices $u, v\in V(G)$ with $uv\notin E(G)$ such that neither $u$ nor $v$ is on an edge with color $1$, meaning that there is no properly colored path joining $u$ and $v$. This is a contradiction.  
Thus we need at least two distinct new colors to make $G$ properly connected. 

Now suppose that we gave an optimal recoloring on $G$. 
By the above observation, we may assume that $G$ contains two edges $e_1, e_2$ with color $1$, $2$, respectively, and other edges have color $0$.  
Assume for the moment that $e_1$ and $e_2$ are independent edges. 
We then consider another different edge-coloring of $G$ by modifying the color of $e_2$ from color $2$ to color $1$. We see that this modified edge-colored $G$ is still properly connected. However, this contradicts the assumption that we gave an optimal recoloring on $G$ before the modification. 
Therefore, we may assume that $e_1$ and $e_2$ share a vertex $w$.  
 Since $m\ge n\ge 4$, there exist two vertices $u, v\in V(G)$ with $uv\notin E(G)$ and $uw, vw\in E(G)$ such that neither $u$ nor $v$ is on an edge with new color. 
We can check that there is no properly colored path joining $u$ and $v$. Since $G$ is now properly connected, this is a contradiction. 

Thus we have $pc_{opt}(G)\ge 5$. 
%Let $a,b,x,y$ be four distinct vertices of $G$ such that $a$ and $b$ belong to a partite set of $G$ and $x$ and $y$ belong to the other partite set of $G$ (thus, $\{ab, xy\}\cap E(G)=\emptyset$). 
Take $a, b\in N$ and $x,y \in M$.  
 Recolor $ax$ and $by$ with color $1$, respectively, and recolor $bx$ with color $2$. It is easy to check that the resulting edge-colored $G$ is properly connected, meaning that $pc_{opt}(G)\le 5$. Consequently, $pc_{opt}(G)=5$. 
 
\end{proof}

Theorem~\ref{kmn} together with Propositions~\ref{f1} and~\ref{P4free} yields the following corollary. 

\begin{cor}
If $G$ is a monochromatic $2$-connected $P_4$-free graph of order at least $9$ then $pc_{opt}(G)\le 5$.  
\end{cor}
%For a graph $G$, let $\alpha'(G)$ be the size of a maximum matching of $G$. 

We can determine the optimal proper connection number for trees as follows. 

\begin{thm}\label{main}
If $T$ is a monochromatic tree of order $n\ge 2$ then $pc_{opt}(T)=n-2-\alpha'(T)+\Delta(T)$. 
\end{thm}

\begin{proof}
The theorem obviously holds for $n=2$. Thus we may assume that $n\ge 3$ and hence $\Delta(T)\ge 2$. Note that, if $G$ is a tree, then the statement that $G$ is properly connected is equivalent to the statement that $G$ is properly colored.

 We first show that $pc_{opt}(T)\ge n-2-\alpha'(T)+\Delta(T)$. Suppose that we gave an edge-coloring on $T$ so that $T$ is properly connected.
Let $T_1$ be a maximal monochromatic subgraph of $T$ with color $0$. Since $T$ is now properly connected, note that $E(T_1)$ forms a matching and hence $|E(T_1)|\le \alpha'(T)$. This implies that we recolored at least $n-1-\alpha'(T)$ edges of $T$. 
Since $T$ is properly connected, note that $T$ contains no monochromatic $P_3$, meaning that we need at least $\Delta(T)-1$ new colors to make $T$ properly connected. Thus we have $pc_{opt}(T)\ge n-2-\alpha'(T)+\Delta(T)$. 

We next show that $pc_{opt}(T)\le n-2-\alpha'(T)+\Delta(T)$. It suffices to show that there is some appropriate edge-coloring on some $n-1-\alpha'(T)$ edges in $T$ with $\Delta(T)-1$ new colors to make $T$ properly connected. To do this, choose a maximum matching $M$  
%$M=\{u_iv_i \in E(T)|\ u_i, v_i\in V(T), 1\le i\le \alpha'(T)\}$
 in $T$ so that $|\{v\in V(T)- V(M)| \  d_T(v)=\Delta(T)\}|$ is as minimum as possible. Now we will give a new color on each edge of $E(T)- M$. Note that $|E(T)- M|=n-1-\alpha'(T)$.  
Put $V(T)- V(M)=\{w_i|\ 1\le i\le l\}$(possibly, an empty set). Since $M$ is a maximum matching, note that $T[\{w_1,\ldots, w_l\}]$ contains no edge.

We now claim that the set $\{v\in V(T)- V(M)| \ d_T(v)=\Delta(T)\}$ is indeed empty. Suppose not, and take a vertex $x\in \{v\in V(T)- V(M)| \ d_T(v)=\Delta(T)\}$. 
Consider a maximal path $P=p_1p_2\ldots p_t$ in $T$ such that $p_1=x$ and $p_{2i}p_{2i+1}\in M$ for every $i\ge 1$. 
By definition, note that $V(P)-\{p_1, p_t\}\subset V(M)$. 
Since $n\ge 3$ and $T[\{w_1,\ldots, w_l\}]$ contains no edge, note that $t\ge 3$.  
By the maximality of $P$, if $p_t\in V(M)$ then $d_T(p_t)=1$. 
Assume for the moment that $p_t\notin V(M)$. This implies that $t$ is an even number. 
%Suppose that there exists a vertex $p_j$ with $j>1$ such that $p_j\in V(T)- V(M)$.  Then by the construction of $P$, this implies that $j$ is an even integer with $j=t$. 
Consequently, $(M- E(P))\cup \{p_{2i-1}p_{2i}|\ 1\le i\le t/2\}$ forms a matching of size greater than $\alpha'(T)$, a contradiction.  
Hence we may assume that $\{p_i|\ 2\le i\le t\}\subset V(M)$, meaning that $t$ is an odd integer and  $d_T(p_t)=1$. 
Then by replacing $M$ by $(M- E(P))\cup \{p_{2i-1}p_{2i}|\ 1\le i\le (t-1)/2\}$, we get a contradiction to the choice of $M$ because $d_T(p_1)=\Delta(T)$ and $d_T(p_t)=1<\Delta(T)$. Thus the claim holds. 
Let $T'$ be the forest obtained from $T$ by deleting all edges of $M$. The above claim implies that $\Delta(T')<\Delta(T)$. 
Hence we can give a proper edge-coloring on $T'$ by using at most $\Delta(T)-1$ new colors. Combining the properly colored $T'$ with $M$, we can make $T$ properly connected, thereby proving that $pc_{opt}(T)\le n-2-\alpha'(T)+\Delta(T)$, and hence $pc_{opt}(T)= n-2-\alpha'(T)+\Delta(T)$, as desired.

\end{proof}

Since any connected graph contains a spanning tree, we obtain the following corollary.

\begin{cor}\label{general}
If $G$ is a monochromatic connected graph of order $n\ge 3$ such that $G\ncong K_n$,  then $\lfloor diam(G)/2\rfloor+1\le pc_{opt}(G)\le \min\{n-2-\alpha'(T)+\Delta(T)|\ T$ is a spanning tree of $G\}\le 2n-4$. 
\end{cor}

The lower bound on Corollary~\ref{general} follows from the fact that we need to recolor at least $\lfloor diam(G)/2\rfloor$ edges on the path joining a pair of vertices with distance $diam(G)$ to make $G$ properly connected. The upper bound on Corollary~\ref{general} can be attained when $G$ is a star.

\section{Some remarks, extension and open problems}

There are many problems together with some extension on the optimal proper connection number of graphs. 

Aside from the case that a connected graph $G$ belongs to some basic family of graphs such as trees or complete bipartite graphs, it might be difficult to find an explicit formula on $pc_{opt}(G)$ for some other family of graphs. Perhaps this could be a challenging problem. 

Let $G$ be a monochromatic connected graph of order $n$. 
If $G$ contains many edges, then it tends to contain a Hamiltonian path, meaning that $pc_{opt}(G)\le \lfloor n/2\rfloor+1$ holds by Corollary~\ref{general} and Proposition~\ref{f1}. It would be an interesting problem to consider what kind of graphs have a constant upper bound on $pc_{opt}(G)$. Also, considering upper bounds on $pc_{opt}(G)$ for sparse graphs would be interesting. For example, what about connected cubic graphs? 

For applications, constructing faster algorithms for giving optimal recolorings in graphs would be important. 
Note that the proofs of our results are constructive. So we can extract a polynomial time algorithm to make $G$ properly connected from there.   

We can also think about some extension in this notion. 
In this paper, we consider the sum of the number of edges and colors when recoloring. 
However, one may simply consider the number of edges for the recoloring. Thus we can define the following function of edge-colored graphs for a monochromatic connected graph $G$. 

 $pc_{opt}'(G):=\min\{p|$ we can make $G$ properly connected\\
\hfill by recoloring $p$ edges of $G \ \}.$\\

Modifying the proofs of our results slightly, we can easily obtain the following counterparts. (Indeed, we have only to skip the argument for counting the number of new colors in the proofs of our previous theorems. So the proofs are omitted.)

\begin{thm}\label{main2}
If $T$ is a monochromatic tree of order $n\ge 2$ then $pc_{opt}'(T)=n-1-\alpha'(T)$. 
\end{thm}

\begin{thm}\label{kmn2}
Let $G$ be a monochromatic complete bipartite graph $K_{m,n}$ such that $m\ge n\ge 2$ and $m+n\ge 9$. Then $pc_{opt}'(G)=2$ for $n=2, 3$, and $pc_{opt}'(G)=3$ for $n\ge 4$.  
\end{thm}

\begin{thm}\label{a22}
If $G$ is a monochromatic connected graph such that $\alpha(G)\le 2$ then $pc_{opt}'(G)\le 2$. 
\end{thm}

When we consider $pc_{opt}'(G)$, we never care about the number of new colors for the recoloring to make $G$ properly connected.  
Conversely, note that, if we consider the number of colors but never care about the number of edges for the recoloring of $G$, then the proper connection number $pc(G)$ can be the counterpart of $pc_{opt}'(G)$.  

We can think about this topic in a more strict manner: For a monochromatic connected graph $G$, one may ask which ordered pair $(p, q)$ with $p+q=pc_{opt}(G)$ gives us the optimal recoloring of $G$. Note that, not all such pairs $(p, q)$ provide us the optimal recoloring of $G$. For this requirement, trivially, we must have $p\ge q\ge pc(G)$, but it is not sufficient in many cases. To describe this new direction more precisely, we define the following. For a monochromatic connected graph $G$, $G$ is $(p,q)$-\textit{feasible} if we can make $G$ properly connected by recoloring $p$ edges with $q$ new colors; in particular, when $p+q=pc_{opt}(G)$, we say that $G$ is $(p, q)$-\textit{optimal feasible}. 

In fact we already had some observation on the $(p,q)$-optimal feasibility for small $p,q$.  To see this, note that, Proposition~\ref{f2} implies that for any non-complete monochromatic connected graph $G$, $G$ is $(1,1)$-optimal feasible if and only if $G$ contains a good edge.  
Moreover, we can extract the following theorem from the proof of Theorem~\ref{kmn}. 

\begin{thm}
Let $G$ be a monochromatic complete bipartite graph $K_{m,n}$ such that $m\ge n\ge 2$ and $m+n\ge 9$. If $n=2, 3$, then $G$ is $(2,2)$-optimal feasible, and if $n\ge 4$, then $G$ is $(3,2)$-optimal feasible. 
\end{thm}   

As we can see from the above argument, our work on the optimal proper connection number could contribute to some problems on $(p, q)$-optimal feasibility in monochromatic connected graphs. The author believe that there will be many interesting problems around this area of study.   

On the other hand, as discussed in \cite{Borozan}, we can consider the ``properly $k$-connected version" in this topic. 
An edge-colored $k$-connected graph $G$ is \textit{properly $k$-connected} if between every pair of distinct vertices, 
there exist $k$ internally-disjoint paths that are properly connected. For a monochromatic $k$-connected graph $G$, we can similarly define the following function. \\

$pc_{opt}^k(G):=\min\{p+q|$ we can make $G$ properly $k$-connected\\
\hfill by recoloring some $p$ edges of $G$ with $q$ new colors$\}.$\\

Details on this function together with the above observation will be discussed elsewhere. \\

\noindent\textbf{Acknowledgments}\\ 

This work was supported by JSPS KAKENHI (No.~15K04979)


\begin{thebibliography}{99}





\bibitem{Borozan}
 Borozan, V., Fujita, S., Gerek, A., Magnant, C., Manoussakis, Y., Montero, L., and Tuza, Z..
{\em Proper connection of graphs,} Discrete Math., {\bf 312} (2012), 2550--2560.

\bibitem{BD}
Brause, C., Doan, T. D., Schiermeyer, I., 
{\em Minimum degree conditions for the proper connection number of graphs,} Gaphs and Combin., {\bf 33} (2017), 833--843. 

\bibitem{Chou}
Chou, W. S., Manoussakis, Y., Megalakaki, O., Spyratos, M., and Tuza, Z., {\em Paths through fixed vertices in edge-colored graphs,} Math. Inform. Sci. Humaines, {\bf 127} (1994), 49--58.

\bibitem{Dorninger}
Dorninger, D., {\em Hamiltonian circuits determining the order of chromosomes,}  Discrete Appl. Math., {\bf 50} (1994) 159--168.
\bibitem{Dorninger2}
Dorninger, D.,  and Timischl, W., {\em Geometrical constraints on Bennet's predictions of chromosome order,} Heredity, {\bf 58} (1987), 321--325.

\bibitem{F}
Fujita, S., Gerek, A., and Magnant, C., 
{\em Proper connection with many colors,} J. Comb., {\bf 3} (2012), 683--693. 


\bibitem{GL}
Gu, R., Li, X., and Qin, Z., 
{\em Proper connection number of random graphs} Theoretical Computer Science, {\bf 609} (2016), 336--343. 

\bibitem{HLQ}
Huang, F., Li, X., Qin, Z., Magnant, C., and Ozeki, K.,
{\em On two conjectures about the proper connection number of graphs,} 
Discrete Math., {\bf 340} (2017), 2217--2222.

\bibitem{HL}
Huang, F., Li, X., and Wang, S.,
{\em Upper bounds of proper connection number of graphs,} J. Comb. Optim., {\bf 34} (2017), 165--173.


\bibitem{Li}
Li, X., and Magnant, C., {\em Properly colored notions of connectivity - a dynamic survey} Theory and Applications of Graphs, {\bf 0}, Article 2. 

\bibitem{LW}
Li, X., Wei, M., and Yue, J., 
{\em Proper connection number and connected dominating sets,} Theoretical Computer Science, {\bf 607} (2015) 480--487. 

\bibitem{West}
 West, D. B., {\em Introduction to Graph Theory}, 2nd edition, Prentice Hall, (2001).


\end{thebibliography}
\end{document}